\documentclass{article}
\author{Nicholas Iammarino}
\title{The Principal Component of the Jets of a Graph}

\usepackage[margin=1in]{geometry}
\usepackage{amssymb, amsmath, mathtools}
\usepackage{amsthm, hyperref, cleveref}
\usepackage{scalerel}
\usepackage{xcolor}
\usepackage{combelow}
\usepackage{graphicx, caption, subcaption, tikz}
\usepackage{hyperref}
\hypersetup{colorlinks=true, urlcolor=blue}

\newcommand{\J}{\mathcal{J}}
\newcommand{\jo}[1]{^{\scaleto{(#1)}{5pt}}}
\newcommand{\p}{\mathfrak{p}}
\newcommand{\q}{\mathfrak{q}}
\newcommand*{\cd}{\fontfamily{pcr}\selectfont}

\newtheorem{theorem}{Theorem}[section]
\newtheorem{corollary}[theorem]{Corollary}
\newtheorem{proposition}[theorem]{Proposition}
\newtheorem{lemma}[theorem]{Lemma}
\newtheorem{definition}[theorem]{Definition}
\theoremstyle{definition}
\newtheorem{example}[theorem]{Example}
\newtheorem{remark}[theorem]{Remark}

\usepackage{setspace}

\begin{document}

\maketitle

\begin{abstract}
	We define the $s$-order principal component of the jets of a graph and give a description of the primary decomposition of its edge ideal in terms of the minimal vertex covers of the base graph.
	As an application, we show the $s$-order principal component of the jets of a cochordal graph is cochordal, and connect this to Fr\"{o}berg's theorem on the linear resolution of edge ideals of cochordal graphs.
	An appendix is provided describing some computations of jets in the computer algebra system \href{https://faculty.math.illinois.edu/Macaulay2/}{Macaulay2}.
\end{abstract}

\section*{Introduction}
In this paper, we explore the idea of a jet scheme as applied to graphs.
The concept of jets was introduced by John Nash in \cite{JN}.
For a set of coordinates $x_1, \dots, x_n$ in affine $n$-space, Nash describes the space of \emph{arcs} parametrized by the family of formal power series $x_i(t)= \sum_{\alpha=0}^\infty x\jo{\alpha}_i t^\alpha$, $1 \leq i \leq n$.
By fixing an upper bound $s \in \mathbb{Z}$ to this power series, we restrict this family to the $s$-jets of the coordinates.
The geometry of jets has been explored by, Ein and Musta\cb{t}\u{a} \cite{EM}, Cornelia Yuen \cite{CY}, Paul Vojta \cite{PV} and many others.
Of particular interest to this paper is the work of Ko\v{s}ir and Sethuraman \cite{KS} which examines the jets of determinantal varieties.
They study the irreducible decomposition of such varieties and discover that they can isolate a \emph{principal} component, which can be described as the closure of the jets of the smooth locus of the base space.
In this paper, we use a similar description to isolate a component of the variety of the edge ideal of a graph (though it turns out not to be irreducible) and explore the connection of this component to minimal vertex covers of the base graph.\\
\\
Jets of graphs are defined by Galetto, Helmick and Walsh in \cite{GHW}, where they examine some properties of graphs, such as their minimal vertex covers and chordality, and whether or not those properties carry over into the jets of a graph in a meaningful way.
We build on that work here by exploring some basic properties of the edge ideals of the jets of graphs.
Starting with a minimal vertex cover $W$ of a graph $G$, we describe the ideal generated by the $s$-jets of elements of $W$, denoted $\J_s(W)$, in terms of the $s$-jets of the edge ideal of $G$, denoted $\J_s(I(G))$ (\cref{vc colon}).
We then show that $\J_s(W)$ is both a minimal prime and a primary component of $\J_s(I(G))$ (\cref{vc ideal prop}).
Following the work of Ko\v{s}ir and Sethuraman\cite{KS}, we give a definition for the $s$-order principal component of a graph (\cref{principal}), describe it in terms of the minimal vertex covers of the base graph (\cref{pc vc theorem}) and show that the principal component is itself a graph (\cref{ideal of pc} and \cref{pc graph}).
Finally, we show that if a graph has an edge ideal with a linear resolution, then so does its principal component (\cref{fberg}).\\
\\
This paper was produced as a master's thesis at Cleveland State University.  It is the culmination of several semesters of work under the guidance and tutelage of Federico Galetto, the fruits of which also include a package for the Macaulay2 language which calculates the jets of various algebraic and geometric objects \cite{jets}, and an accompanying paper \cite{GI} describing the functionality of the package.
We include here an appendix discussing some of this work.
The author is eternally grateful to Dr. Galetto for his patience and wisdom, as well as the ungrudging generosity he showed with his time.

\section{Background}\label{bg}

\subsection*{Graphs}
A graph $G$ can be described as a pair of sets $\{ V(G), E(G) \}$ where $V(G)$ is the set of vertices of $G$ and $E(G)$ is a set of pairs of vertices which form the edges of $G$.
Throughout this paper, all graphs are assumed to be simple and connected.
The following definitions, which can be found in \cite{VT}, allow the construction of an ideal corresponding to the edge set and set of minimal vertices of a graph.

\begin{definition}\label{graphdefs}
	Let $G$ be a graph with vertex set $V(G) = \{ x_1, \dots , x_n \}$.
	\begin{enumerate}
		\item A subset $W \subseteq V(G)$ is a vertex cover if $W \cap e \neq \varnothing$ for all $e \in E(G)$.  A vertex cover is a minimal vertex cover if no proper subset of $W$ is a vertex cover.
		\item The edge ideal corresponding to $G$ is the monomial ideal
		\begin{equation*}
			I(G)= \langle x_ix_j \ | \ \{x_i, x_j \} \in E(G) \rangle \subseteq R = k[x_1, \dots , x_n] \rangle
		\end{equation*}
	\end{enumerate}
\end{definition}

\begin{example}\label{ex:graph}
		Let $R$ be the polynomial ring $k[v_1, \dots, v_5]$ over a field $k$.
	\begin{figure}[h!]
		\centering{}
		\caption{}
		\label{ex:graph:fig}
		\begin{subfigure}{.45\textwidth}
			\centering
			\begin{tikzpicture}[scale=.4, auto]
		    \node[circle,draw](1) at (5,5) {$v_1$};
		    \node[circle,draw](2) at (5,10) {$v_2$};
		    \node[circle,draw](3) at (10,5) {$v_3$};
		    \node[circle,draw](4) at (5,0) {$v_4$};
		    \node[circle,draw](5) at (0,5) {$v_5$};

		    \draw[-](1) to node {}(2);
		    \draw[-](1) to node {}(3);
		    \draw[-](1) to node {}(4);
				\draw[-](1) to node {}(5);
		  \end{tikzpicture}
			\caption{$G_1$}
			\label{ex:graph:fig:a}
		\end{subfigure}
		\begin{subfigure}{.45\textwidth}
			\centering
			\begin{tikzpicture}[scale=.4, auto]
		    \node[circle,draw](1) at (5,10) {$v_1$};
		    \node[circle,draw](2) at (10,6) {$v_2$};
		    \node[circle,draw](3) at (8,0) {$v_3$};
		    \node[circle,draw](4) at (2,0) {$v_4$};
		    \node[circle,draw](5) at (0,6) {$v_5$};

		    \draw[-](1) to node {}(3);
		    \draw[-](1) to node {}(4);
		    \draw[-](2) to node {}(4);
				\draw[-](2) to node {}(5);
				\draw[-](3) to node {}(5);
				\draw[dashed](4) to node {}(5);
		  \end{tikzpicture}
			\caption{$G_2$}
			\label{ex:graph:fig:b}
		\end{subfigure}
	\end{figure}
  Then the vertices of the graphs in \cref{ex:graph:fig} are indeterminates of $R$ and we can form their edge ideals.
	In \cref{ex:graph:fig:a} we have a star on five vertices.
	The edge ideal of this graph is given by its four edges, $I(G_1) = \langle v_1v_2, v_1v_3, v_1v_4, v_1v_5 \rangle$.
	Its minimal vertex covers consist of the set $\{ v_1 \}$ containing only the center vertex, and the set $\{ v_2, v_3, v_4, v_5 \}$ containing all of the outer vertices.
	For $G_2$ in \cref{ex:graph:fig:b} the solid edges yield the monomial ideal $I(G_2) = \langle v_1v_3, v_1v_4, v_2v_4, v_2v_5, v_3v_5 \rangle$ along with five minimal vertex covers corresponding to the non-adjacent triples of vertices $\{ v_1,v_2,v_3 \}, \{ v_1,v_2,v_5 \}, \{ v_1, v_4, v_5 \},  \{ v_2,v_3,v_4 \}$ and $ \{ v_3,v_4,v_5 \}$.
	If we include the dashed edge connecting $v_4$ and $v_5$, updating the edge ideal is simply a matter of adding the ideal generated by that edge: $I(G_2) + \langle v_4v_5 \rangle$.
	For the vertex covers, we lose the set $\{ v_1,v_2,v_3 \}$ as it does not account for the new edge and is therefore not a vertex cover.
	Notice also, that the minimal vertex covers $\{ v_1, v_4, v_5 \}$ and $\{ v_3, v_4, v_5 \}$ contain both of the vertices of the new edge, but we can still find edges containing $v_4$ and $v_5$ whose opposite vertex is not in the cover.
\end{example}

For a geometric interpretation of these edge ideals, we can turn to the following description of their decomposition.
A graph $G$ with vertex set $V(G) = \{ v_1, \dots , v_n \}$ and minimal vertex covers $W_1, \dots, W_t$ has edge ideal $I(G) = \bigcap_{i=1}^t \langle W_i \rangle$ \cite[corollary 1.35]{VT}.
We can write $\langle W_i \rangle = \langle v_{i_1}, \dots, v_{i_r} \rangle$ where the $v_{i_j}$ are vertices of $G$ contained in the vertex cover $W_i$.
The $\langle W_i \rangle$ are clearly prime, and since $I(G)$ is radical and completely described by their intersection, we have its decomposition into associated primes, and therefore a decomposition of the variety $\mathcal{V}(I(G))$ into a union of irreducible components.
Furthermore, the components of $\mathcal{V}(I(G))$  corresponding to the ideals $\langle W_i \rangle$ form coordinate subspaces of $\mathbb{A}^n_k$.
So the variety corresponding to a graph on $n$ vertices is a union of coordinate subspaces of $\mathbb{A}^n$.

\begin{example}\label{graphgeo}
	Consider the following graphs with vertices in $k[x,y,z]$.
	Let $G_1$ be the path of length two with edge set $E(G_1) = \{ \{ x,z \}, \{ y,z \} \}$ and $G_2$ be the $3$-cycle with edge set $E(G_2) = \{ \{ x,y \}, \{ y,z \}, \{ x,z \} \}$.
	Then we have two graphs on three vertices, which we can think of in terms of their corresponding objects in affine $3$-space.
	$G_1$ has two minimal vertex covers, $\{ x, y \}$ and $\{ z \}$, so the variety corresponding to $G_1$ is $\mathcal{V}(I(G_1)) = \mathcal{V}(x,y) \cup \mathcal{V}(z) $ or the union of the $z$-axis and the $xy$-plane.
	Since the minimal vertex covers of $G_2$ are all possible pairs of vertices, we have $\mathcal{V}(I(G_2)) = \mathcal{V}(x,y) \cup \mathcal{V}(x,z) \cup \mathcal{V}(y,z)$ or the union of the three coordinate axes.
\end{example}

In addition to this geometric interpretation, \cite[corollary 1.35]{VT} shows that, for any graph $G$, each associated prime of $I(G)$ is minimal, i.e. $I(G)$ has no embedded primes.
This fact is guaranteed by the minimality of the vertex covers we are using to describe its decomposition.
For a vertex cover $W$, we denote by $W^C$ its complement, which consists of all vertices of $G$ not contained in $W$.
We record two properties of vertex covers in the following remark.

\begin{remark}\label{vertex covers remark}
	Let $G$ be a graph.
	\begin{enumerate}
		\item If $W$ is a minimal vertex cover, for any given $x \in W$ there exists $y \in W^C$ such that $\{ x, y \}$ is an edge of $G$.
		To see this, consider all edges of $G$ containing $x$ which we can label $ \{ x,y_1 \}, \dots, \{ x, y_n \}$.
		If $y_i$ is in $W$ for all $i$ then $x$ is clearly redundant, contradicting the minimality of $W$.
		\item If $W_1, \dots, W_r$ is the set of minimal vertex covers of $G$, a vertex $x$ cannot belong to  $W_\alpha$ for all $\alpha$.
		This is a direct result of \cite[corollary 1.35]{VT} since $I(G) = \langle W_1 \rangle \cap \cdots \cap \langle W_r \rangle$ and $x \in W_\alpha$ for all $\alpha$ implies $x \in I(G)$ which is impossible.
	\end{enumerate}
\end{remark}

\subsection*{Jets}
Let $R$ be a polynomial ring over a field $k$.
For a positive integer $s$, define the truncation ring $T_s:= k[t]/\langle t^{s+1} \rangle$.
Then a homomorphism $\phi_s: R \longrightarrow T_s$ sends variables of $R$ to degree $s$ polynomials in $T_s$.
Explicitly
\begin{align*}
	\phi_s:
		x_i &\mapsto x\jo{0}_i + x\jo{1}_i t + \cdots + x\jo{s}_i t^s\\
		c &\mapsto c, \ \ \ c \in k
\end{align*}
where the $x\jo{l}_i$ take values in $k$ for $0 < l \leq s$.
Any polynomial $f \in R$ can be considered a function on the variables of $R$; applying $\phi_s$ gives
\begin{equation*}
	\phi_s(f(x_1,\dots,x_n)) = f(x\jo{0}_1 + x\jo{1}_1t + \cdots + x\jo{s}_1t^s, \dots, x\jo{0}_n + x\jo{1}_nt + \cdots + x\jo{s}_nt^s).
\end{equation*}

\begin{example}\label{jets ex}
	Let $R=k[x,y,z]$, $s=2$ and $f= x^2y$.  Then
	\begin{align*}
		\phi_2(f) &= (x\jo{0} + x\jo{1}t + (x\jo{2})t^2)^2(y\jo{0} + y\jo{1}t + y\jo{2}t^2)\\
		&= ((x\jo{0})^2 + 2x\jo{0}x\jo{1}t + (2x\jo{0}x\jo{2} + (x\jo{1})^2)t^2)(y\jo{0} + y\jo{1}t + y\jo{2}t^2)\\
		&=(x\jo{0})^2y\jo{0} + (2x\jo{0}y\jo{0}x\jo{1} + (x\jo{0})^2y\jo{1})t + ((x\jo{0})^2y
		\jo{2} + 2x\jo{0}y\jo{0}x\jo{2} + 2x\jo{0}x\jo{1}y\jo{1} + y\jo{0}(x\jo{1})^2)t^2
	\end{align*}
\end{example}
In \cref{jets ex}, we see clearly that the image of $f$ under $\phi_s$ is a polynomial in $t$ whose coefficients we can treat as polynomials in the variables $x\jo{l}_i$ for $1\leq i \leq n$ and $0 \leq l \leq s$ which exist in their own polynomial ring. Denote by
\begin{equation*}
	\J_s(R) = k[x\jo{0}_1, \dots, x\jo{s}_1, \dots, x\jo{0}_n, \dots, x\jo{s}_n]
\end{equation*}
the polynomial ring in these variables.
So if $R$ is a polynomial ring in $n$ variables, then $\J_s(R)$ is a polynomial ring in $n(s+1)$ variables.
Now if we have an ideal $I=\langle f_1, \dots, f_r \rangle$ of $R$, we can restrict $\phi_s$ to the quotient $R/I$.
Then, as in \cref{jets ex}, $\phi_s$ sends each generator $f_i$ to some polynomial in $t$ and we can write
\begin{equation}\label{alphas}
	\phi_s|_{R/I}(f_i) = \alpha\jo{0}_i + \alpha\jo{1}_i t + \cdots + \alpha\jo{s}_i t^s
\end{equation}
where the $\alpha\jo{l}_i$ are polynomials in $\J_s(R)$ for $0 \leq l \leq s$.
Since $\phi_s$ is a homomorphism, its restriction to $R/I$ must send the generators of $I$ to zero.
Therefore $\alpha\jo{l}_i = 0$ for each $l$, and we have a new set of relations defining an ideal of $\J_s(R)$ which we denote
\begin{equation*}
	\J_s(I) = \langle \alpha\jo{0}_1, \dots, \alpha\jo{s}_1, \dots, \alpha\jo{0}_r, \dots, \alpha\jo{s}_r \rangle.
\end{equation*}
If we consider the truncation ring $T_s$ as a vector space,  the generators of $\J_s(I)$ collected from \cref{alphas} can be broken up into independent components $\alpha\jo{m}$ each corresponding to the basis element $t^m$.
It is easy to see that for sufficiently large $s$ (precisely $s \geq m$), the element $\alpha\jo{m}$ becomes fixed and no longer depends on $s$.
If we take $X$ to be the variety with coordinate ring $R/I$, then the $s$-jets of $X$ form the variety $\J_s(X)$ with coordinate ring $\J_s(R)/\J_s(I)$\cite{GS}.
We refer to $\J_s(R)$ and $\J_s(I)$ as the $s$-jets of $R$ and $I$ respectively, and the set $\{ x\jo{0}_i, x\jo{1}_i, \dots \}$ as the jets variables of $x_i$.
Finally, it follows from definitions that $\J_0(\bullet) \cong \bullet$ for any applicable object.

\begin{example}\label{cord subspace}
	Let $X$ to be the coordinate subspace of $\mathbb{A}^n_k$ defined by the ideal $I=\langle x_{i_1}, \dots, x_{i_r} \rangle \subset k[x_1, \dots x_n]$.
	Then the $s$-jets of $I$ given by $\J_s(I) = \langle x_{i_1}\jo{0}, \dots, x_{i_1}\jo{s}, \dots, x_{i_r}\jo{0}, \dots x_{i_r}\jo{s} \rangle$ define a variety $\J_s(X)=\mathcal{V}(\J_s(I))$ which is itself a coordinate subspace of $\mathbb{A}^{n(s+1)}_k$.
\end{example}

If we take a positive integer $m < s$, there is a natural inclusion $\psi_{m,s}: \J_m(R) \longrightarrow \J_s(R)$ which embeds jets of elements of $R$ into a higher order jets ring.
In terms of the geometry, for an affine variety $X$ we have the canonical projection $\pi_{s,m}^X: \J_s(X) \longrightarrow \J_m(X)$ which projects points in $\mathbb{A}_k^{n(s+1)}$ down to points in $\mathbb{A}_k^{n(m+1)}$ \cite[section 2]{EM}.
To simplify notation, $\pi_s^X : \J_s(X) \longrightarrow \J_0(X) \cong X$ and we can omit the superscript when it is clear from context.
Applying the inclusion $\psi_{m,s}$ to $\alpha\jo{m}$ allows each of these components to exist, in a sense, in every jets ring of higher order.
With this in mind, we can view the $s$- jets of an ideal $I=\langle f_1, \dots, f_r \rangle$ as being built incrementally as a sum of ideals of lower order.
In other words, $\J_s(I) = \J_{s-1}(I) + \langle \alpha\jo{s}_1, \dots, \alpha\jo{s}_r \rangle$ gives a recursive definition of the $s$-jets of $I$ with respect to the image under $\phi_s$ of its generators \cite[section 3]{GS}.

In \cite{GHW} the authors explore the jets of the edge ideals of a graph.
To define the jets of a graph, they make the following observations.
Edge ideals are squarefree monomial ideals.
The jets of a monomial ideal do not, in general, form a monomial ideal.
They do, however, form an ideal whose radical is squarefree and monomial \cite[theorem 3.1]{GS}, and since the base ideal is quadratic, the resulting jets ideal is as well \cite[theorem 2.2]{GHW}.
Therefore the radical of the $s$-jets of an edge ideal is the edge ideal of a graph, which the authors define as the jets of the base graph.
We restate their definition here.

\begin{definition}\cite[section 2]{GHW}
	Let $G$ be a graph with edge ideal $I(G)$.  Then the $s$-jets of $G$, denoted $\J_s(G)$, is the graph defined by the ideal $\sqrt{\J_s(I(G))}$.
\end{definition}

They also give a lemma (\cite[lemma 2.4]{GHW}) describing the edge set of the jets of a graph, which we illustrate in the following example.

\begin{example}\label{jets edges}
	Let $G$ be a graph and $\{ x, y \}$ an edge of $G$.
	Then $xy$ is a generator of $I(G)$ and, as in \cref{alphas}, we can find the corresponding generators of $\J_s(I(G))$:
	\begin{align*}
		\phi_s(xy) = &(x\jo{0} + x\jo{1}t + \cdots + x\jo{s}t^s)(y\jo{0} + y\jo{1}t + \cdots + y\jo{s}t^s)\\
			= &(x\jo{0}y\jo{0})\ \  +\\
			 	&(x\jo{0}y\jo{1} \ + x\jo{1}y\jo{0})t\ \  +\\
				&(x\jo{0}y\jo{2} \ + x\jo{1}y\jo{1} + x\jo{2}x\jo{2})t^2\ \  +\\
				&\vdots\\
				&(x\jo{0}y\jo{s} \ + x\jo{1}y\jo{s-1} + \cdots + x\jo{s}y\jo{0})t^s
	\end{align*}
	Since $\J_s(G)$ is defined by the radical of $\J_s(I(G))$, we apply \cite[theorem 2.1]{GS} to extract the edges of $\J_s(G)$ (see \cref{M2}) which are the terms of each of the coefficient polynomials of the powers of $t$.
\end{example}

We state the lemma for completeness.

\begin{lemma}\cite[lemma 2.4]{GHW}\label{jets edges lemma}
	Let $G$ be a graph and let $x,y$ be distinct vertices of $G$. For every non-negative integer $s$, the set $\{ x\jo{i}, y\jo{j} \}$ is an edge in $\J_s(G)$ if and only if $\{ x, y \}$ is an edge of $G$ and $i + j \leq s$.
\end{lemma}

\Cref{ex:jetsgraphs:fig} gives a visual example of the jets of a path on three vertices.

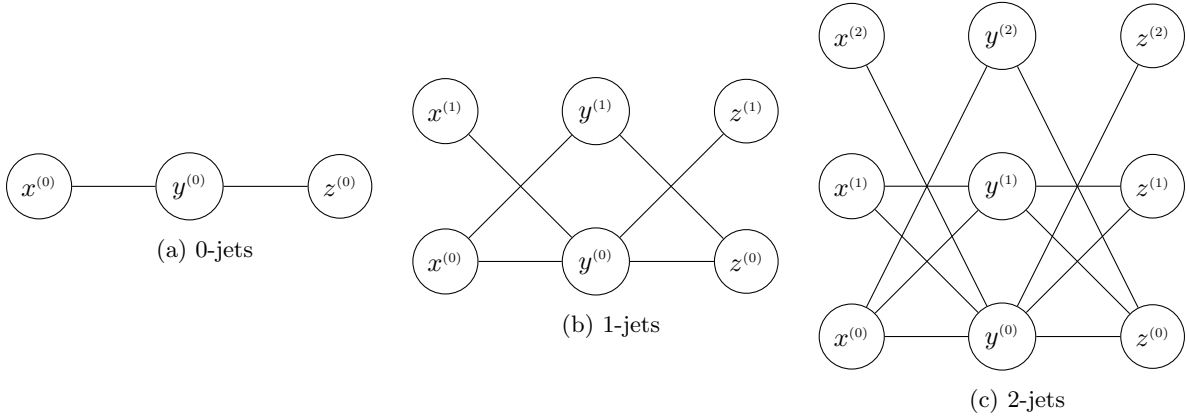
\begin{figure}[h!]
	\centering
	\caption{The jets of the path of length two}
	\label{ex:jetsgraphs:fig}
	\begin{subfigure}{.32\textwidth}
		\centering
		\begin{tikzpicture}[scale=.4, auto]
			\node[circle,draw](x0) at (0,0) {$x\jo{0}$};
			\node[circle,draw](y0) at (5,0) {$y\jo{0}$};
			\node[circle,draw](z0) at (10,0) {$z\jo{0}$};

			\draw[-](x0) to node {}(y0);
			\draw[-](y0) to node {}(z0);
		\end{tikzpicture}\hspace{\fill}
		\caption{$0$-jets}
		\label{ex:jetsgraphs:fig:a}
	\end{subfigure}
	\begin{subfigure}{.32\textwidth}
		\centering
		\begin{tikzpicture}[scale=.4, auto]
			\node[circle,draw](x0) at (0,0) {$x\jo{0}$};
			\node[circle,draw](y0) at (5,0) {$y\jo{0}$};
			\node[circle,draw](z0) at (10,0) {$z\jo{0}$};
			\node[circle,draw](x1) at (0,5) {$x\jo{1}$};
			\node[circle,draw](y1) at (5,5) {$y\jo{1}$};
			\node[circle,draw](z1) at (10,5) {$z\jo{1}$};

			\draw[-](x0) to node {}(y0);
			\draw[-](y0) to node {}(z0);
			\draw[-](x0) to node {}(y1);
			\draw[-](z0) to node {}(y1);
			\draw[-](y0) to node {}(x1);
			\draw[-](y0) to node {}(z1);
		\end{tikzpicture}\hspace{\fill}
		\caption{$1$-jets}
		\label{ex:jetsgraphs:fig:b}
	\end{subfigure}
	\begin{subfigure}{.32\textwidth}
		\centering
		\begin{tikzpicture}[scale=.4, auto]
			\node[circle,draw](x0) at (0,0) {$x\jo{0}$};
			\node[circle,draw](y0) at (5,0) {$y\jo{0}$};
			\node[circle,draw](z0) at (10,0) {$z\jo{0}$};
			\node[circle,draw](x1) at (0,5) {$x\jo{1}$};
			\node[circle,draw](y1) at (5,5) {$y\jo{1}$};
			\node[circle,draw](z1) at (10,5) {$z\jo{1}$};
			\node[circle,draw](x2) at (0,10) {$x\jo{2}$};
			\node[circle,draw](y2) at (5,10) {$y\jo{2}$};
			\node[circle,draw](z2) at (10,10) {$z\jo{2}$};

			\draw[-](x0) to node {}(y0);
			\draw[-](y0) to node {}(z0);
			\draw[-](x0) to node {}(y1);
			\draw[-](z0) to node {}(y1);
			\draw[-](y0) to node {}(x1);
			\draw[-](y0) to node {}(z1);
			\draw[-](y2) to node {}(x0);
			\draw[-](y2) to node {}(z0);
			\draw[-](y0) to node {}(x2);
			\draw[-](y0) to node {}(z2);
			\draw[-](y1) to node {}(z1);
			\draw[-](x1) to node {}(y1);
		\end{tikzpicture}\hspace{\fill}
		\caption{$2$-jets}
		\label{ex:jetsgraphs:fig:b}
	\end{subfigure}
\end{figure}

\section{Jets from a Vertex Cover}
In this section we will discuss ideal quotients and refer to the following definition:

\begin{definition}\cite[definition 4.4.5]{CLO}\label{def colon}
Let $I$, $J$ be ideals in a polynomial ring $R$.  Then the ideal quotient and saturation of $I$ with respect to $J$ are, respectively
\begin{enumerate}
	\item $I:J = \langle f \in R \ | \ fg \in I \text{ for all } g \in J \rangle$
	\item $I:J^{\infty} = \langle f \in R \ | \ \text{ for all } g \in J \text{ there is an integer } N \geq 0 \text{ such that } fg^N \in I  \rangle$
\end{enumerate}
\end{definition}

If $J$ is principal, we can view the quotient of $I$ with $J$ as the quotient of $I$ with the polynomial $f \in R$ that generates $J$.
To get an idea of what this operation does, we can think of it as a method to ``peel off'' factors from elements of an ideal.
So if $f$ divides $h=fg \in I$ then $g$ will appear in the quotient; if $f$ does not divide $h$, then certainly $fh \in I$ and $h$ appears in the quotient unaffected.
As a consequence of this, if $f=f1_R \in I$ then $1_R$ is in the quotient and we have $I:f = R$.
One property of the ideal quotient which will be applied in this section is illustrated in the following remark.

\begin{remark}\cite[exercise 1.12]{AM}\label{colon ass}
	For and ideal $I \subseteq R$ and polynomials $a,b \in R$,
	\begin{align*}
		(I:a):b = \{ g \in R \ | \ gb \in \{f \in R \ | \ fa \in I \} \} = \{ g \in R \ | \ gab \in I \} = I:ab
	\end{align*}
\end{remark}

Let $G$ be a graph with vertices in a polynomial ring $R$ and edge ideal $I(G)$ and let $W$ be a minimal vertex cover of $G$.
For any non-negative integer $s$, denote by $\J_s(W)$ the subset of $\J_s(R)$ given by $\{ x_i^{(j)} \ | \ 0 \leq j \leq s \text{ and } x_i \in W \}$.
Then $\J_s(W)$ is a minimal vertex cover of $\J_s(G)$ \cite[proposition 5.3]{GHW}.
Fix an element $x \in W$.
Since $W$ is minimal, we can find $y \in W^C$ such that $\{ x, y \}$ is an edge of $G$, and their product $xy$ is therefor an element of $I(G)$.

\begin{lemma}\label{colon lemma}
	Let $G$ be a graph and fix $W$, a minimal vertex cover of $G$.  Given any edge $\{ x,y \} $ of $G$ with $x \in W$ and $y \in W^C$, $x\jo{i}$ is an element of the quotient $\J_s(I(G)):(y\jo{0})^{s+1}$ for all $i \leq s$.
\end{lemma}

\begin{proof}
	We prove the claim by induction on the order of jets.
	The base case is evident from the fact that $\J_0(I(G))$ and $I(G)$ are isomorphic as rings.
	Consider the following element of $\J_s(I(G))$ constructed from its $s-$order generator (see \cref{jets edges}):
	\begin{equation}\label{pdcolon}
		(y\jo{0})^s \cdot (x\jo{0}y\jo{s} + x\jo{1}y\jo{s-1} + \cdots + x\jo{s}y\jo{0}).
	\end{equation}
	Assuming $x\jo{i} \in \J_{s-1}(I(G)):(y\jo{0})^s$ for all $0 \leq i \leq s-1$, we have $x\jo{i}(y\jo{0})^s \in \J_{s-1}(I(G))$ for all such $i$.
	Since the image of $\J_{s-1}(I(G))$ under inclusion is contained in $\J_s(I(G))$ (\cref{bg}) each term of \cref{pdcolon} with the exception of the last exists as a monomial in $\J_s(I(G))$ implying that the sum
	\begin{equation}\label{sumcolon}
		\sum^{s-1}_{i=0} x\jo{i}y\jo{s-1-i}(y\jo{0})^s
	\end{equation}
	is also an element of $\J_s(I(G))$.
	Taking the difference of \cref{pdcolon} and \cref{sumcolon} we find $x\jo{s}(y\jo{0})^{s+1} \in \J_s(I(G))$ implying $x\jo{s}$ is an element of $\J_s(I(G)):(y\jo{0})^{s+1}$.
\end{proof}

\begin{remark}\label{colon lemma rem}
	More generally, for any monomial $f \in \J_s(R)$ with $(y\jo{0})^{s+1}$ as factor, $x\jo{s}f \in \J_s(I(G))$ (see \cref{colon ass}).
	We use this fact to find an expression for the ideal generated by the $s$-jets of a vertex cover.
\end{remark}

\begin{theorem}\label{vc colon}
	For any vertex cover $W$ of a graph $G$,
	\begin{align*}
		\langle \J_s(W) \rangle = \J_s(I(G)):f^{\infty}
	\end{align*}
	where $f$ is the product of the elements of $\J_0(W^C)$.
\end{theorem}

\begin{proof}
	Let $f$ be as stated in the claim.  Then for any $0 \leq i \leq s$, \cref{colon lemma} guarantees $x\jo{i} \in \J_s(I(G)):f^{s+1}$ for all $x \in W$.
	Therefor $\J_s(I(G)):f^{s+1} \subseteq \langle \J_s(W)\rangle$.
	For the opposite containment take an arbitrary generator $x\jo{i}$ of $\langle \J_s(W) \rangle$.
	From the proof of \cref{colon lemma}, $x\jo{i}(y\jo{0})^{s-1}$ is a monomial in $\J_s(I(G))$ for any $i \leq s$.
	By construction, $y^{s-i}$ divides $f^{s+1}$ and \cref{colon lemma rem} leads to the conclusion that $x\jo{i} \in \J_s(I(G)):y^{s-i} \subseteq \J_s(I(G)):f^{s+1}$.
	We conclude that $\langle \J_s(W) \rangle = \J_s(I(G)):f^{s+1}$, which implies $\langle \J_s(W):f \rangle = \J_s(I(G)):f^{s+2}$.
	But by construction $x\jo{i}$ does not divide $f$ for any $x\jo{i} \in \J_s(W)$ so the colon does not change the ideal of the vertex cover, and $\J_s(I(G)):f^{s+1} = \J_s(I(G)):f^{s+2}$ which is a sufficient condition for saturation \cite[proposition 4.4.9 (ii)]{CLO}.
\end{proof}

These ideals generated by the elements of a minimal vertex cover appear in the primary decomposition of the edge ideal of a graph.
In general, given an ideal $I$ of a polynomial ring, we can decompose it into an intersection of primary ideals.
All of these primary components corresponding to minimal primes of $I$ are uniquely determined \cite[corollary 4.11]{AM}, and can be recovered from a given minimal prime using localization.
This is known for Modules in general and is discussed in \cite{SLA, MCRT} among others.
We could not find a reference stating this result explicitly for ideals of polynomial rings, so we present it here as a lemma.
We will use the notation $R_\p$ for the localization of the ring $R$ at the prime ideal $\p$, $\phi_\p$ for the homomorphism from $R$ to $R_\p$ defined by $r \xmapsto{\phi_p} \frac{r}{1_R}$, and $IR_\p$ for the image of $I$ under $\phi_\p$.

\begin{lemma}\label{q from p}
	Let $R=k[x_1,\dots,x_n]$ and $I \subseteq R$ be an ideal with associated prime $\p \supseteq I$.
	If $\p$ is minimal, it corresponds to a primary component $\q$ of $I$ with $\q = \phi_\p^{-1}(IR_\p)$.
\end{lemma}

\begin{proof}
	Let $R$ and $I$ be as above and write the minimal primary decomposition $I = \q_1 \cap \cdots \cap \q_r$ with $\sqrt{I} = \p_1 \cap \cdots \cap \p_r$ where $\p_i = \sqrt{\q_i}$.
	We can arrange this decomposition so that $\p_1$ is minimal, and set $\p = \p_1$ and $\q = \q_1$ which forces $\q \subseteq \p$.
	\\
	Take the intersection of the remaining primary components and denote $I^* = \bigcap_{i=2}^r \q_i$ so that $I = \q \cap I^*$.
	Let $f \in \q$ and $g \in I^*$.
	Then their product $fg$ is in $I$ and $\phi_\p(fg)=\frac{fg}{1_R} \in IR_\p$.
	If we choose $g$ such that $g \notin \p$, then $\frac{1}{g}$ is an element of $R_\p$ and we have $\frac{f}{1_R} \in IR_\p$ which implies $f \in \phi_\p^{-1}(IR_\p)$ (provided such a $g$ exists).
	Assume for the sake of contradiction that there is no such $g$.
	Then $I^* \subseteq \p$ and $\q_i \subseteq \p$ for some $i$\cite[proposition 1.11]{AM}, so $\sqrt{\q_i} \subseteq \sqrt{\p} = \p$, which contradicts the minimality of $\p$.
	Therefore $\q \subseteq \phi_\p^{-1}(IR_\p)$.\\
	\\
	For the opposite containment, take $f \in \phi_\p^{-1}(IR_\p)$.
	Then $\phi_\p(f) \in IR_\p$ and we have the relation $\frac{f}{1} = \frac{h}{g}$ where $h \in I$ and $g \notin \p$, implying $h = fg$.
	Since $fg$ is an element of $I$, it is also an element of the $\p$-primary ideal $\q$, which contains no power of $g$ (as $g \notin \p$) and must therefor contain $f$.
	We conclude that $\phi_\p^{-1}(IR_\p) \subseteq \q$ and the lemma is proved.
\end{proof}

Using \cref{q from p} we can describe some of the components of the decomposition of the jets of an edge ideal.

\begin{theorem}\label{vc ideal prop}
	Let $G$ be a graph and $W$ a minimal vertex cover of $G$.  Then $\langle \J_s(W) \rangle$ is both a minimal prime and a primary component of $\J_s(I(G))$.
\end{theorem}

\begin{proof}
	Let $W= \{ x_1, \dots, x_n \}$ be a minimal vertex cover of a graph $G$.
	Then
	\begin{equation*}
		\p = \langle \J_s(W) \rangle = \langle x\jo{0}_1, \dots, x\jo{0}_n, \dots, x\jo{s}_1, \dots, x\jo{s}_n \rangle
	\end{equation*}
	is clearly prime.\\
	\\
	For the first claim, to show $\p$ is a minimal prime of $\J_s(I(G))$, it is sufficient to show that it is a minimal prime of $\sqrt{\J_s(I(G))}$.
	Since $I(G)$ is a square free monomial ideal, \cite[theorem 2.1]{GS} guarantees that $\sqrt{\J_s(I(G))}$ is also a monomial ideal with a generating set obtained by collecting all of the terms of the generators of $\J_s(I(G))$.
	Furthermore, since $I(G)$ is the edge ideal of a graph, we can describe this generating set:
	\begin{align*}
		\sqrt{\J_s(I(G))} = \langle x\jo{i}y\jo{j} \ | \ i+j \leq s \text{ and } xy \in I \rangle
	\end{align*}
	where the $x\jo{i} \in \J_s(W)$ are generators of $\p$ \cite[lemma 2.4]{GHW}.
	Since $W$ is a vertex cover of $G$, $\J_s(W)$ is a vertex cover of $\J_s(G)$ \cite[proposition 5.3]{GHW} and any $x\jo{i}y\jo{j}$ in its edge ideal is also an element of the ideal $\p$.
	We therefore have the containment $I( \J_s(G)) = \sqrt{\J_s(I(G))} \subseteq \p$.
	Now assume there exists some prime ideal $\p'$ with $\sqrt{\J_s(I(G))} \subseteq \p' \subseteq \p$.
	Given an arbitrary generator $xy$ of $\sqrt{\J_s(I(G))}$ we have $xy \in \p$ and $xy \in \p'$, where at least one of $x$ or $y$ is an element of $\p$.
	In the case that only one (say $x$) is in $\p$ then by containment and the primality of $\p'$ we must have $x$ in $\p'$.
	If both $x$ and $y$ are generators of $\p$ then the minimality of $W$ guarantees the existence of elements $x'$ and $y'$ such that $xx', yy' \in \J_s(I(G))$ and $x', y' \notin \p$.
	Then by the previous argument both $x$ and $y$ must be elements of $\p'$.
	Therefor $\p \subseteq \p'$ and $\p$ is a minimal prime of both $\sqrt{\J_s(I(G))}$ and $\J_s(I(G))$.\\
	\\
	For the second claim, since $\p$ is minimal, it corresponds to a unique primary ideal $\q$ with $\sqrt{\q} = \p$ and $\q$ a primary component of $\J_s(I)$\cite[corollary 4.11]{AM}.
	As in \cref{q from p}, we can construct the local ring $\J_s(R)_\p$ and the corresponding homomorphism $\phi_\p$.
	Let $x\jo{i} \in \p$.
	Since $\p$ originates from a minimal vertex cover of a graph, we can find an element $y\jo{j} \in \J_s(R)$ such that $x\jo{i}y\jo{j} \in \J_s(I(G))$ with $y\jo{j} \notin \p$.
	Then $\phi_\p(x\jo{i}y\jo{j}) = \frac{x\jo{i}y\jo{j}}{1}$ is in the extension $\J_s(I(G))\J_s(R)_\p$ which implies $\frac{x\jo{i}}{1} = \frac{1}{y\jo{j}} (\frac{x\jo{i}y\jo{j}}{1})$ is as well.
	Therefore $x\jo{i} \in \q$ by \cref{q from p}, and $\p = \langle \J_s(W) \rangle \subseteq \q$.
	Naturally $\q \subseteq \p$, and we have equality, showing that $\langle \J_s(W) \rangle$ is a primary component of $\J_s(I(G)).$
\end{proof}

\section{The principal component of the edge ideal of a graph}
Having a description of the variety associated to the jets of a graph, it seems natural to examine its irreducible components.
In \cite{KS}, K\u{o}sir and Sethuraman study the jets of determinantal varieties of an $m \times n$ matrix with entries in an algebraically closed field $k$ by mapping the entries of $M$ into the truncation ring $k[t]/ \langle t^{s+1} \rangle$ (\cref{bg}).
One particular component of such a variety (if it turns out to be reducible) is referred to as the \emph{principal component} and can be described as the ``closure of the set of jets supported over the smooth points of the base [variety]''\cite{GJS}.
We can extend this description to the variety defined by the edge ideal of a graph.

\begin{definition}[Principal Component]\label{principal}
	Let $G$ be a graph in $n$ vertices with edge ideal $I(G)$.  Let $X = \mathcal{V}(I(G))$.  Then the $s$-order principal component of $G$ is the Zariski closure of the $s$-jets of the smooth locus of $X$.
\end{definition}

We note that, in this case, the name \emph{principal component} can be misleading since, as it turns out, when applied to the edge ideal of a graph, we get variety that is not irreducible in general.
With $X$ as described in \cref{principal}, denote by $X_{smooth}$ and $X_{sing}$ the smooth and singular loci of $X$ respectively.
Since $X$ is a variety corresponding to a squarefree quadratic monomial ideal, it is simply a union of coordinate subspaces of $\mathbb{A}^n$ and we can easily describe its singular points.

\begin{proposition}\label{open components}
	Let $G$ be a graph in $n$ vertices, with corresponding variety $X= \mathcal{V}(I(G))$.
	Then a point $p \in X$ is singular if and only if it lies on the intersection of two or more irreducible components of $X$.
\end{proposition}

\begin{proof}
	Let $X$ be described by the irredundant irreducible decomposition $F_1 \cup \dots \cup F_r$.  In \cref{bg} we saw that each of these components corresponds to a minimal vertex cover of $G$, and therefore represents a coordinate subspace of $X$.
	Take $F^* = \bigcup_{i \neq j} F_i \cap F_j$.  Given a point $p \in F^*$ it follows immediately from \cite[theorem 9.6.8]{CLO} that $p \in X_{sing}$.

	For the opposite containment, choose $p=(p_1,\dots,p_n) \in X$ such that $p$ lies on one and only one component of $X$, say $p \in F_{\alpha}$, which is a coordinate subspace defined by the elements of some minimal vertex cover $W_{\alpha} = \{ x_{\alpha_1}, \dots, x_{\alpha_t} \}$.
	$X$ is also a variety defined by the edges of graph;
	if $e_{i,j} = x_ix_j$ corresponds to $\{ x_i, x_j \} \in E(G)$, then $X = \mathcal{V}(e_{i,j} \ | \ \{ x_i, x_j \} \in E(G))$.
	We use both of these descriptions to show $p$ is smooth by following the procedure, outlined in \cite[9.6]{CLO}.
	First, since $p$ is contained only in the irreducible component $F_\alpha$, $\dim_p(X) = \dim_p(F_\alpha)$ \cite[definition 9.6.6]{CLO}.
	Next, we can describe the tangent space of $X$ at $P$
	\begin{equation*}
		T_p(X) = \mathcal{V}(d_p(e_{i,j}) \ | \ e_{i,j} \in \mathcal{I}(X))
	\end{equation*}
	where
	\begin{equation}\label{linear part}
		d_p(e_{i,j}) = \frac{\partial x_ix_j}{\partial x_i}(p)(x_i-p_i) + \frac{\partial x_ix_j}{\partial x_j}(p)(x_j-p_j) = p_jx_i + p_ix_j - 2p_ip_j
	\end{equation}
	is the linear part of $e_{i,j}$ at $p$ \cite[definition 9.6.1]{CLO}.
	Since $p$ lies in a component described by the minimal vertex cover $W$, each edge of $G$ has at least one vertex in $W$, so at least one of $p_i$ or $p_j$ must be zero.
	Furthermore, the minimality of $W$ guarantees we can find an edge $\{ x_i, x_j \}$ whose second vertex is not in $W$, that is either $p_i=0$ or $p_j=0$ but not both.
	Then \cref{linear part} shows $T_p(X)$ is defined by  $\{ p_jx_i \ | \ x_i \in W\}$, which, working over a field, is identical to the generating set of $W$.  Therefore $\dim T_p(X) = \dim_p(X)$ and $p$ is non-singular, showing $X_{sing} \subseteq F^*$ by contraposition.
\end{proof}

From the decomposition $X=F_1 \cup \cdots \cup F_r$ we can write the smooth locus of $X$ as
\begin{equation}\label{smooth open union}
	X_{smooth} = X \backslash X_{sing}
		= (F_1 \cup \cdots \cup F_n) \backslash X_{sing}
		= F_1 \backslash X_{sing} \cup \cdots \cup F_n \backslash X_{sing}
		= U_1 \cup \cdots \cup U_n
\end{equation}
where the $U_i$ are proper subsets of the $F_i$.
This is evident since each component of $X$ is a coordinate subspace of $\mathbb{A}^n$ contained in $X$, which implies that $0$ is an element of the singular locus of $X$ and each of its irreducible components.
Each $U_i$ of \cref{smooth open union} is therefore open in its corresponding component $F_i$ and, by \cref{open components} open in the whole of $X$ as well.
The following theorem gives a description of the $s$-order principal component of a graph.

\begin{theorem}\label{pc vc theorem}
	Let $G$ be a graph with minimal vertex covers $W_1, \dots, W_m$.
	Then the $s$-order principal component of $G$ is precisely the union $\mathcal{V}(\langle \J_s(W_1)\rangle) \cup \dots \cup \mathcal{V}(\langle \J_s(W_m) \rangle)$.
\end{theorem}

\noindent To prove this result, we will appeal to the following lemma of Ein and Musta\cb{t}\u{a}.

\begin{lemma}\cite[lemma 2.3]{EM}\label{EM open sets}
	If $U \subseteq X$ is an open subset and if $\J_s(X)$ exists, then $\J_s(U)$
exists and $\J_s(U) = \pi^{-1}_s (U)$.
\end{lemma}

\begin{proof}[Proof of \cref{pc vc theorem}]
Let $X = \mathcal{V}(I(G))$.
Then \cite[corallary 1.35]{VT} guarantees a decomposition
\begin{equation*}
	X = \mathcal{V}(\langle W_1 \rangle) \cup \dots \cup \mathcal{V}(\langle W_m \rangle).
\end{equation*}
Letting $F_i = \mathcal{V}(\langle W_i \rangle)$, we can use \cref{smooth open union} to write
\begin{equation*}
	\J_s(X_{smooth}) = \J_s(U_1 \cup \cdots \cup U_n),
\end{equation*}
describing the $s$-jets of the smooth locus of $X$.
Because $X_{smooth}$ is open in $X$ \cite[theorem 5.3]{RH}, we can apply \cref{EM open sets}:
\begin{equation}\label{proj union1}
	\J_s(X_{smooth}) = \pi^{-1}_s(U_1 \cup \cdots \cup U_n) = \pi^{-1}_s(U_1) \cup \cdots \cup \pi^{-1}_s(U_n),
\end{equation}
and since each $U_i$ is open in $X$, applying \cref{EM open sets} again yields
\begin{equation}\label{proj union2}
	\pi^{-1}_s(U_1) \cup \cdots \cup \pi^{-1}_s(U_n) = \J_s(U_1) \cup \cdots \cup \J_s(U_n).
\end{equation}
Paired with the fact that the Zariski closure of a union is the union of its Zariski closures \cite[lemma 4.4.3, (iii)]{CLO}, \cref{proj union1} and \cref{proj union2} give an expression for the principal component of $X$ as a union of closures of jets of open sets:
\begin{equation*}
	\overline{\J_s(X_{smooth})} = \overline{\J_s(U_1)} \cup \cdots \cup \overline{\J_s(U_n)}.
\end{equation*}
Returning to the decomposition of $X$, consider an arbitrary component $F_i$ of $X$ along with the following facts:
\begin{enumerate}
	\item \emph{$\J_s(U_i)$ is open in $\J_s(F_i)$} since the projection $\pi^{F_i}_s: \J(F_i) \longrightarrow F_i$ is continuous, and $U_i$ is an open subset of $F_i$.  That is, the preimage $(\pi^{F_i}_s)^{-1}(U_i) = \J_s(U_i)$ is open in $\J_s(F_i)$.
	\item \emph{$\J_s(F_i)$ is irreducible} since it is a coordinate subspace (see \cref{cord subspace}).
\end{enumerate}
Taken together, they show $\J_s(U_i)$ is Zariski dense in $\J_s(F_i)$ \cite[proposition 4.5.13]{CLO}.
Therefore the Zariski closure of $\J_s(U_i)$ is $\J_s(F_i)$ and we conclude
\begin{equation*}
	\overline{\J_s(X_{smooth})} = \J_s(F_1) \cup \dots \cup \J_s(F_m) = \mathcal{V}(\langle \J_s(W_1)\rangle) \cup \dots \cup \mathcal{V}(\langle \J_s(W_m) \rangle)
\end{equation*}
\end{proof}

With this description of the prinipal component, we can find its corresponding ideal.

\begin{proposition}\label{ideal of pc}
	The ideal of the $s$-order principal component of a graph is given by
	\begin{equation*}
		\langle x\jo{i} y\jo{j} \ | \ \{x,y \} \in E(G) \rangle
	\end{equation*}
\end{proposition}

\begin{proof}
	Let $G$ be a graph with vertices in $R$ and minimal vertex covers $W_1, \dots, W_m$.
	Since each $\langle \J_s(W_\alpha) \rangle$ is radical, by \cref{pc vc theorem} the ideal of the $s$-order principal component of $G$ is
	\begin{equation*}
		\langle \mathcal{V}(\langle \J_s(W_1)\rangle) \cup \dots \cup \mathcal{V}(\langle \J_s(W_m) \rangle) \rangle= \bigcap_{1 \leq \alpha \leq m} \langle \J_s(W_\alpha) \rangle.
	\end{equation*}
	Let $x\jo{i}y\jo{j} \in \tilde{I}$ with $\tilde{I}=\langle x\jo{i} y\jo{j} \ | \ \{x,y \} \in E(G) \rangle$ as above.
	Then for any given $\alpha \in \{ 1, \dots, m \}$, either $x \in W_\alpha$ or $y \in W_\alpha$ (or both) which implies $x\jo{i} \in \J_s(W_\alpha)$ for all $0 \leq i \leq s$ or $y\jo{j} \in \J_s(W_\alpha)$ for all $0 \leq j \leq s$.
	Since this is true for all $\alpha$, $x\jo{i}y\jo{j} \in \bigcap_{1 \leq \alpha \leq m} \langle \J_s(W_\alpha) \rangle$ and $\tilde{I} \subseteq \bigcap_{1 \leq \alpha \leq m} \langle \J_s(W_\alpha) \rangle$.\\
	\\
	For the opposite containment we apply \cref{vc colon} to the intersection, which yields
	\begin{equation*}
		\bigcap_{1 \leq \alpha \leq m} \langle \J_s(W_\alpha) \rangle = \bigcap_{1 \leq \alpha \leq m} \J_s(I(G)):f_\alpha^{\infty}
	\end{equation*}
	where $f_\alpha$ is the product of the elements of $\J_0(W_\alpha^C)$.
	Since each $\langle \J_s(W_\alpha) \rangle$ is radical, so is the intersection \cite[proposition 4.3.16]{CLO} and we can write
	\begin{equation*}
		\bigcap_{1 \leq \alpha \leq m} \J_s(I(G)):f_\alpha^{\infty} = \sqrt{\bigcap_{1 \leq \alpha \leq m} \J_s(I(G)):f_\alpha^{\infty}} = \bigcap_{1 \leq \alpha \leq m} \sqrt{\J_s(I(G)):f_\alpha^{\infty}}.
	\end{equation*}
	Using \cite[proposition 4.4.9 (iii) and proposition 4.4.13 (i)]{CLO} we can further reduce this expression to conclude
	\begin{equation*}
		\bigcap_{1 \leq \alpha \leq m} \langle \J_s(W_\alpha) \rangle = \sqrt{\J_s(I(G))} : \langle f_1, \dots, f_m \rangle = I(\J_s(G)) : \langle f_1, \dots, f_m \rangle,
	\end{equation*}
	which is the quotient of an edge ideal.
	Now choose a monomial $g \in I(\J_s(G)) : \langle f_1, \dots, f_m \rangle$.
	Then for each $f_\alpha$, $gf_\alpha$ is a monomial in $I(\J_s(G))$, and therefor a multiple of some edge monomial of $\J_s(G)$, say $gf_\alpha = \rho_\alpha (x\jo{i}y\jo{j})$ where $\rho_\alpha \in \J_s(R)$ and $\{ x,y \}$ is an edge of the base graph $G$.
	If $x\jo{i}$ divides $f_\alpha$ for all $\alpha$ then $x \in W_\alpha^C$ for all $\alpha$ which implies $y \in W_\alpha$ for all $\alpha$.
	This contradicts \cref{vertex covers remark}.
	Therefore $x\jo{i}$ does not divide $f_\alpha$ for some $\alpha$ which implies $x\jo{i}$ must divide $g$.
	By a similar argument $y\jo{j}$ must divide $g$, and we conclude that the product $x\jo{i}y\jo{j}$ divides $g$ and $g \in \tilde{I}$, proving the containment $\bigcap_{1 \leq \alpha \leq m} \langle \J_s(W_\alpha) \rangle \subseteq \tilde{I}$.
\end{proof}

\begin{corollary}\label{pc graph}
	The ideal of the principal component of a graph is itself the edge ideal of a graph.
\end{corollary}

To end this section we give the following summary of notation and definitions for the $s$ order principal component of relevant objects:\\
\begin{remark}\label{pc notation}
	Let $G$ be a graph with vertices in $R$, edge ideal $I$, vertex covers $W_1, \dots, W_r$ and corresponding variety $V= \mathcal{V}(I)$.  Then
	\begin{enumerate}
		\item $PC_s(V) = \bigcup_{1 \leq \alpha \leq r} \mathcal{V}(\langle \J_s(W_\alpha)) \rangle$
		\item $PC_s(I) = \langle x\jo{i} y\jo{j} \ | \ \{x,y \} \in E(G) \rangle$
		\item $PC_s(G) = \{ V(\J_s(G)), \tilde{E}(G) \}$ where $\tilde{E}(G) = \{ \{ x\jo{i}, y\jo{j} \} \ | \ \{ x,y \} \in E(G) \}$
	\end{enumerate}
\end{remark}

\section{Chordal graphs and Fr\"{o}berg's theorem}

For any graph, a cycle of length $l$ is a sequence of vertices denoted $(x_1\ x_2 \cdots \ x_l\ x_1)$ which form a closed path in $G$.
That is, for each adjacent pair $x_i \ x_j$ in the sequence, $\{ x_i, x_j \}$ is an edge of $G$.
If for some non-adjacent $x_i$ and $x_j$ in the sequence, $\{ x_i, x_j \}$ is an edge of $G$, it is said to be a chord of the cycle.
A minimal cycle is one which has no chords \cite[definition 2.10]{VT}.
We will consider only cycles where the $x_i$ are distinct, as any cycle containing a repeated variable (other than $x_1$) can be split into two cycles in an obvious way.
In this section we will use the following two properties of a graph.
\begin{definition}\cite[section 2]{VT}\label{complement and chordal}
	Let $G$ be a graph.  Then
	\begin{enumerate}
		\item The complementary graph of $G$ is given by $G^C = \{ V(G), E(G)^C \}$ where $E(G)^C = \{ \{ x, y \} \in V(G) \ | \ \{ x, y \} \notin E(G) \}$.
		\item $G$ is chordal if it has no minimal cycles of length greater than three.
	\end{enumerate}
\end{definition}

If a graph has a complement which is chordal, then it is itself referred to as \emph{cochordal}.
This quality is a condition of Fr\"{o}berg's theorem, which states that a graph is cochordal if and only if its edge ideal has a linear resolution \cite[theorem 2.13]{VT}.
We might ask how the jets of a graph interact with this theorem.
A description of the edges of the complement of the jets of a graph can be derived from \cite[lemma 2.4]{GHW}.
For a graph $G$, we can write the edge set of $\J_s(G)$ as
\begin{equation}\label{JG}
	E(\J_s(G)) = \{ \{ x\jo{i},y\jo{j} \} \subseteq V(\J_s(G)) \ | \ \{x, y \} \in E(G) \text{ and } i + j \leq s \}.
\end{equation}
then we can obtain the edge set of its complement by negating the conditions of the comprehention in \cref{JG}:
\begin{align}\label{JG complement}
	E(\J_s(G)^C) = \{ \{ x\jo{i},y\jo{j} \} \subseteq V(\J_s(G))\ | \ &\{x, y \} \in E(G) \text{ and } i + j > s, \nonumber \\
	\text{or } &\{ x, y \} \notin E(G) \}.
\end{align}
Notice that $x\jo{a}$ and $x\jo{b}$ are distinct vertices of $\J_s(G)$ when $a \neq b$, so this definition includes all edges of the form $\{ x\jo{a}, x\jo{b} \}$ with $a \neq b$ which correspond to a single vertex of the base graph, and therefore cannot correspond to one of its edges.

\begin{example}
	Let $G$ be the path of length three with edges $\{ x, y \}$, $\{ y, z \}$ and $\{ z, w \}$.
	Its complement $G^C$ has edges $\{ y, w \}$, $\{ x, w \}$, and $\{ x, z \}$ which is also a path of length three.
	Since $G^C$ has no cycles at all it cannot have a minimal cycle of length greater than three so it is chordal.
	$G$ is therefore cochordal.
	Now consider $\J_1(G)^C$.
	From \cref{JG complement} we see that $\{ x\jo{0}, z\jo{1} \}$, $\{ x\jo{0}, w\jo{0} \}$, $\{ y\jo{1}, w\jo{0} \}$, and $\{ z\jo{1}, w\jo{0} \}$ are all edges of $\J_1(G)^C$.
	Therefore it contains the cycle $(x\jo{0}\ z\jo{1}\ y\jo{1}\ w\jo{0} x\jo{0})$.
	But $\{ x\jo{0}, y\jo{1} \}$ and $\{ w\jo{0}, z\jo{1} \}$ are edges of $\J_s(G)$ and cannot be elements of the complement.  So we have found a minimal cycle of length four and the $1$-jets of $G$ are not cochordal.
\end{example}

\begin{theorem}\label{pc cochordal}
	Let $G$ be a cochordal graph.
	Then for ever integer $s \geq 0$, $PC_s(G)$ is cochordal.
\end{theorem}

\begin{proof}
	Using \cref{pc notation} we can describe the edge set of the complement of the principal component of $G$ as
	\begin{equation*}
				E(PC_s(G)^C)= \{ \{ x\jo{i}, y\jo{j} \} \ | \ \{x, y \} \notin E(G) \}.
	\end{equation*}
	Notice from the definitions that $PC_s(G^C)$ is a subset of $PC_s(G)^C$.
	The containment is not reversible however, since $PC_s(G)^C$ contains edges $\{ x\jo{i}, y\jo{j} \}$ for which $x = y$ and $i \neq j$, but $PC_s(G^C)$ is restricted by the edge set of $G^C$.
	Now let $K = (x_1 \ x_2 \ \cdots \ x_l \ x_1)$ be a cycle of length $l > 3$ in $G^C$.  Then each adjacent pair $\{ x_i, x_j \}$ of the cycle is an edge of $G^C$ and, following the same indexing, each pair $\{ x_i\jo{a}, x_j\jo{b} \}$,  $0 \leq a,b \leq s$, is an edge of $PC_s(G^C)$.
	So $K$ gives rise to a family of cycles in $PC_s(G^C)$ which we denote
	\begin{equation*}
		\bar{K} = \{ (x_1\jo{a_1} \ x_2\jo{a_2} \ \cdots \ x_l\jo{a_l} \ x_1\jo{a_1}) \ | \ a_1, \dots, a_n \in \{0,1,\dots,s \} \}.
	\end{equation*}
	Since $G$ is cochordal, there is an edge $\{ x_{i^*}, x_{j^*} \}$ of $G^C$ where $x_{i^*}$ and $x_{j^*}$ are non-adjacent in $K$.
	Then $\{ x\jo{a}_{i^*}, x\jo{b}_{j^*} \}$ is an edge of $PC_s(G)^C$ for all $a,b$ with $0 \leq a,b \leq s$.
	So for an arbitrary cycle $K$ of $G^C$, every cycle in  $\bar{K}$ (which is a cycle of $PC(G^C) \subset PC(G)^C$) has a chord.\\
	Now let $\kappa= (x\jo{a_1}_{c_1} \ x\jo{a_2}_{c_2} \ \cdots \ x\jo{a_l}_{c_l} \ x\jo{a_1}_{c_1})$ be an arbitrary cycle in $PC_s(G)^C$ with $l > 3$.
	If the $x_{c_i}$ are distinct, then $\kappa$ corresponds to a cycle $K$ of $G^C$ which implies $\kappa \in \bar{K}$ and therefore has a chord.
	If the $x_{c_i}$ are not distinct (say $x_{c_i} = x_{c_j}$ for some $i \neq j$, $0 \leq i,j \leq l$ we must consider two cases:
	\begin{enumerate}
		\item If $x\jo{a_i}_{c_i}$ and $x\jo{a_j}_{c_j}$ are non-adjacent in $\kappa$, they form a chord since $PC_s(G)^C$ contains all edges of the form $\{ x\jo{a}, x\jo{b} \}$ with $a \neq b$ and $x \in V(G)$.
		\item If $x\jo{a_i}_{c_i}$ and $x\jo{a_j}_{c_j}$ are adjacent in $\kappa$, we can arrange the indices so that $\kappa$ contains a path $(x\jo{a_i}_{c_i} \ x\jo{a_j}_{c_i} \ x\jo{a_{j+1}}_{c_{j+1}})$.
		Since $\{ x\jo{a_j}_{c_j}, x\jo{a_{j+1}}_{c_{j+1}} \}$ is an edge in $PC_s(G)^C$ and $x_{c_i} = x_{c_j}$, $\{ x\jo{a_i}_{c_i}, x\jo{a_{j+1}}_{c_{j+1}} \}$ must also be an edge of $PC_s(G)^C$ by definition.
		Therefore $\kappa$ contains a chord.
	\end{enumerate}
	We conclude that $PC_s(G)^C$ is chordal, and $PC_s(G)$ is therefore cochordal.
\end{proof}
\begin{corollary}\label{fberg}
	It follows from Fr\"{o}berg's theorem that if $I(G)$ has a linear resolution, so does $PC_s(I(G))$.
\end{corollary}

\begin{example}
	Let $G$ be the complete bipartite graph on the five vertices $x_1, \dots x_5$ with edge ideal
	\begin{equation*}
		I(G) = \langle x_{1}x_{4},\,x_{2}x_{4},\,x_{3}x_{4},\,x_{1}x_{5},\,x_{2}x_{5},\,x_{3}x_{5} \rangle.
	\end{equation*}
	The complement of this graph is the union of the path of length one and the $3$-cycle, which is chordal, so $G$ is cochordal, and by Fr\"{o}berg's theorem has a linear resolution with Betti table:\\
	\begin{equation*}
		\begin{matrix}
						& 0 & 1 & 2 & 3 & 4\\
				 \text{total:}
						& 1 & 6 & 9 & 5 & 1\\
				 0: & 1 & . & . & . & .\\
				 1: & . & 6 & 9 & 5 & 1
		 \end{matrix}
	\end{equation*}
	Then $\J_1(I(G))$ has a linear resolution with Betti table:
	\begin{equation*}
		\begin{matrix}
	            & 0 & 1 & 2 & 3 & 4 & 5 & 6 & 7 & 8 & 9\\
	        \text{total:}
	            & 1 & 24 & 96 & 194 & 246 & 209 & 120 & 45 & 10 & 1\\
	         0: & 1 & . & . & . & . & . & . & . & . & .\\
	         1: & . & 24 & 96 & 194 & 246 & 209 & 120 & 45 & 10 & 1
	  \end{matrix}
	\end{equation*}
	and $\J_2(I(G))$ has a linear resolution with Betti table:
	\begin{equation*}
		\begin{matrix}
      & 0 & 1 & 2 & 3 & 4 & 5 & 6 & 7 & 8 & 9 & 10 & 11 & 12 & 13 & 14\\
    \text{total:}
      & 1 & 54 & 351 & 1224 & 2871 & 4920 & 6399 & 6426 & 5004 & 3003 & 1365 & 455 & 105 & 15 & 1\\
      0: & 1 & . & . & . & . & . & . & . & . & . & . & . & . & . & .\\
      1: & . & 54 & 351 & 1224 & 2871 & 4920 & 6399 & 6426 & 5004 & 3003 & 1365 & 455 & 105 & 15 & 1
  	\end{matrix}
	\end{equation*}

\end{example}

We have seen that, in some cases, the principal component preserves some information about the resolution of an edge ideal.
It is natural to ask if we can use this fact to predict or recover any information.
For example, is there a connection between the Betti numbers of a cochordal graph and those of its $s$-order principal component.
The concept of jets could also be defined for simplicial complexes via their Stanley-Reisner ideals.
We can extend the definition of the principal component as well by replacing the edge ideal of a graph with the Stanley-Reisner ideal of a simplicial complex.
Some initial investigation indicates that this process may preserve some of the homological information of the base complex.
\appendix

\section{Calculating Jets with Macaulay2}\label{M2}

As a semester project, the author and his advisor constructed a package \cite{jets} for the Macaulay2 language \cite{M2} to work with jets in polynomial rings.
In this appendix we present some of the methods used to calculate the jets of a few objects in Macaulay2.

\subsection*{The radical of the $s$-jets of a monomial ideal}
The Jets package offers a method for calculating the radical of the jets of a monomial ideal based on \cite[theorem 3.1]{GS}.
The theorem states that, given a monomial ideal $I \subseteq R$, the $s$-jets of $I$, $\J_s(I) \subseteq \J_s(R)$, has a radical which is a squarefree monomial ideal.
The statement of the theorem describes the monomial generators of $\sqrt{\J_s(I)}$ as
\begin{equation*}
	\sqrt{x\jo{i_1}_1 x\jo{i_2}_1 \cdots x\jo{i_{a_1}}_1 x\jo{i_{1+a_1}}_2 \cdots x\jo{i_{a_1+a_2}}_2 \cdots x\jo{i_{a_1 + \cdots + a_r}}_r}
	\text{ where } \sum i_j \leq s \ \cite[\text{theorem 3.1}]{GS}
\end{equation*}
ranging over the minimal generators $x_1^{a_1} \cdots x_r^{a_r}$ of $I$.
This is a combinatorial description of the \emph{terms} of the coefficients of our polynomial in $t$ (which we labeled $\alpha_i$) as illustrated in \cref{jets edges}.
Using this idea, we define the function {\cd jetsRadical} in Macaulay2 which returns the radical of the jets of a monomial ideal without calculating a Gr\"{o}bner basis.
As we have seen, for each monomial generator of $I$, the $s$-jets of $I$ has corresponding generators for each power of $t$ up to $s$.
We can isolate the terms of each these generators by applying the {\cd terms} function to the result.
This gives us a list of monomials, whose radicals we find by taking the {\cd support}  of each (yielding a list of variables of $\J_s(R)$ present non-trivially in the monomial), and taking the {\cd product} of the elements of the result.
Running this process over each of the generators of $\J_s(I)$ gives a set of generators for $\sqrt{\J_s(I)}$.
It should be noted that this set is not necessarily a minimal generating set.

\subsection*{The principal component of an ideal}

The Jets package also provides a method {\cd principalCompoent} which, given any ideal $I$ of a polynomial ring, returns an ideal $I'$ such that $\mathcal{V}(I')$ is the Zariski closure of the smooth locus of $\mathcal{V}(I)$ embedded in the space of $s$-jets.  This ideal is a sort of generalization of the ideal of the variety defined in \cref{principal}.
The process of calculating it relies on \cite[theorem 4.4.10]{CLO} and we summarize here its description given in the documentation of \cite{jets}.
Denoting by $A$ the ideal of $X_{sing}$ and $J$ the ideal of $\J_s(X)$, the theorem shows that
\begin{equation*}
	\overline{\J_s(X_{smooth})} = \overline{\J_s(X) \backslash X_{sing}} = \overline{\mathcal{V}(J) \backslash \mathcal{V}(A)} = \mathcal{V}(J : A^{\infty})
\end{equation*}
This method returns the ideal $J : A^{\infty}$.
To accomplish this we call on the {\cd jetsProjection} method, which is an encoding of the canonical projection described in \cref{bg}.
Mapping the {\cd singularLocus} of the input ideal to its $0$-jets via the natural isomorphism, we apply {\cd jetsProjection} to get the result as $s$-jets.
The function {\cd saturate} applied to the jets of the input ideal and the projected ideal returns the desired ideal.

\bibliography{Iammarino_jpcgraphbbl}
\bibliographystyle{plain}
\end{document}